\documentclass{article}

\usepackage{fullpage}
\usepackage{amsthm}
\usepackage{amsmath}
\usepackage{enumerate}
\usepackage{amssymb}

\usepackage[utf8]{inputenc} 

\usepackage{graphicx}
\usepackage{caption}
\usepackage{subcaption}

\usepackage[justification=centering]{caption}

\newtheorem{theorem}{Theorem}[section]
\newtheorem{proposition}[theorem]{Proposition}

\newtheorem{corollary}[theorem]{Corollary}

\newtheorem{definition}{Definition}[section]
\newtheorem{example}[theorem]{Example}
\newtheorem{remark}[theorem]{Remark}

\begin{document}

	\title{Osculating-type Ruled Surfaces in the Euclidean 3-space}
	
	\author{Onur Kaya$^{1}$, Tanju Kahraman$^{1}$, Mehmet Önder$^{2}$ \\[3mm] {\it $^{1}$ Manisa Celal Bayar University, Department of Mathematics, 45140, Manisa, Turkey} \\{E-mails: onur.kaya@cbu.edu.tr, tanju.kahraman@cbu.edu.tr} \\ {\it $^{2}$ Delibekirli Village, Tepe Street, No:63, 31440, Kırıkhan, Hatay, Turkey} \\{E-mail: mehmetonder197999@gmail.com} \date{}}
	
	\maketitle
		
	\begin{abstract}
	In the present paper, a new type of ruled surfaces called osculating-type (OT)-ruled surface is introduced and studied. First, a new orthonormal frame is defined for OT-ruled surfaces. The Gaussian and the mean curvatures of these surfaces are obtained and the conditions for an OT-surface to be flat or minimal are given. Moreover, the Weingarten map of an OT-ruled surface is obtained and the normal curvature, the geodesic curvature and the geodesic torsion of any curve lying on surface are obtained. Finally, some examples related to helices and slant helices are introduced.
	\end{abstract}

	

	\section{Introduction}
	In the study of fundamental theory of curves and surfaces, the special ones of these geometric topics have been of significant value because of satisfying some particular conditions. In the curve theory, the most famous one of such special curves is general helix for which the tangent vector of the curve always makes a constant angle with a constant direction. The necessary and sufficient condition for a curve to be a general helix is that the ratio of the second curvature $\tau$ to the first curvature $\kappa$ is constant i.e., $\tau / \kappa$ is constant along the curve \cite{Barros}. If the principal normal vector of a curve makes a constant angle with a constant direction, then that curve is called slant helix and the necessary and sufficient condition for a curve to be a slant helix is that the function $\sigma (s) = \left( \frac{\kappa^2}{\left( \kappa^2 + \tau^2 \right)^{3/2}} \left( \frac{\tau}{\kappa} \right)' \right) (s)$ is constant \cite{IzuTakeSlant}.
	
	In the surface theory, the surfaces constructed by the simplest way are important. The well-known example of such surfaces is ruled surface which is generated by a continuous movement of a line along a curve. These surfaces have a wide use in technology and architecture \cite{Emmer}. Furthermore, some special types of these surfaces have particular relationships with helices and slant helices \cite{IzuTakeSlant,IzuTakeSpec,IzuTakeGeom,Izumiyaetal}. In \cite{OnderSlant}, \"{O}nder considered the notion of "slant helix" for ruled surfaces and defined slant ruled surfaces by the property that the components of the frame along the striction curve of ruled surface make constant angles with fixed lines. He has proved that helices or slant helices are the striction curves of developable slant ruled surfaces. Also, he has defined a new kind of ruled surfaces called general rectifying ruled surface for which the generating line of the surface always lies on the rectifying plane of base curve and he has given many properties of such surfaces \cite{OnderRect}.
	
	This study introduces a new type of ruled surfaces called osculating-type (OT)-ruled surfaces. First, a new orthonormal frame and new curvatures for OT-ruled surfaces are obtained and many properties of the surface are given by considering the new frame and its curvatures. Later, the Gaussian curvature $K$ and the mean curvature $H$ of OT-ruled surfaces are given. The set of singular points of such surfaces are introduced and some differential equations characterizing special curves lying on the surface are obtained. Finally, some examples related to helix and slant helix are given.
	
	\section{Preliminaries}
	A ruled surface in $\mathbb{R}^3$ is constructed by a continuous movement of a straight line along a space curve $\alpha$. For an open interval $I \subset \mathbb{R}$, the parametric equation of a ruled surface is given by $\varphi_{(\alpha, q)} (s,u) : I \times \mathbb{R} \rightarrow \mathbb{R}^3$, $\vec{\varphi}_{(\alpha, q)} (s,u) = \vec{\alpha} (s) + u \vec{q} (s)$ where $q: I \rightarrow \mathbb{R}^3$, $\lVert \vec{q} \hspace{1pt} \rVert = 1$ is called director curve and $\alpha: I \rightarrow \mathbb{R}^3$ is called the base curve of the surface $\varphi_{(\alpha, q)}$. The straight lines of the surface defined by $u \rightarrow \vec{\alpha} (s) + u \vec{q} (s)$ are called rulings \cite{IzuTakeSpec}. The ruled surface $\varphi_{(\alpha, q)}$ is called cylindrical if $\vec{q}\hspace{2pt}' = 0$  and non-cylindrical otherwise where $\vec{q}\hspace{2pt}' = \frac{d\vec{q}}{ds}$ \cite{KargerNovak}. A curve $c$ lying on $\varphi_{(\alpha, q)}$ with property that $\langle \vec{c}\hspace{2pt}', \vec{q} \hspace{2pt}' \rangle = 0$ is called striction line of $\varphi_{(\alpha, q)}$. The parametric representation of striction line is given by
	\begin{equation} \label{strictionline}
		\vec{c} (s) = \vec{\alpha} (s) - \frac{\langle \vec{c}\hspace{2pt}' (s), \vec{q} \hspace{2pt}' (s) \rangle}{\langle \vec{q}\hspace{2pt}' (s), \vec{q} \hspace{2pt}'  (s) \rangle} \vec{q} (s)
	\end{equation}
	The striction line is geometrically important because it is the locus of special points called central points for which considering a common perpendicular between two constructive rulings, the foot of common perpendicular on the main ruling is a central point \cite{KargerNovak}.
	
	The unit surface normal or Gauss map $U$ of the ruled surface $\varphi_{(\alpha, q)}$ is defined by
	\begin{equation*}
		\vec{U} (s,u) = \frac{\frac{\partial \vec{\varphi}_{(\alpha, q)}}{\partial s} \times \frac{\partial \vec{\varphi}_{(\alpha, q)}}{\partial u}}{\left\| \frac{\partial \vec{\varphi}_{(\alpha, q)}}{\partial s} \times \frac{\partial \vec{\varphi}_{(\alpha, q)}}{\partial u} \right\|} .
	\end{equation*}
	If $\frac{\partial \vec{\varphi}_{(\alpha, q)}}{\partial s} \times \frac{\partial \vec{\varphi}_{(\alpha, q)}}{\partial u} = 0$ for some points $({{s}_{0}},{{u}_{0}})\in \,I\times \mathbb{R}$ then, such points are called singular points of ruled surface $\varphi_{(\alpha, q)}$. Otherwise, they are called regular points. The surface $\varphi_{(\alpha, q)}$ is called developable if the unit surface normal $U$ along any ruling does not change its direction. Otherwise, $\varphi_{(\alpha, q)}$ is called non-developable or skew. A ruled surface $\varphi_{(\alpha, q)}$ is developable if and only if $\det (\vec{\alpha }', \vec{q}, \vec{q}\hspace{2pt}') = 0$ holds \cite{KargerNovak}.
	
	The unit vectors $\vec{h}= \vec{q} \hspace{2pt}' / \left\| \vec{q} \hspace{2pt}' \right\|$ and $\vec{a} = \vec{q} \times \vec{h}$ are called central normal and central tangent of $\varphi_{(\alpha, q)}$, respectively. Then, the orthonormal frame $\left\{ \vec{q}, \vec{h}, \vec{a} \right\}$ is called the Frenet frame of ruled surface $\varphi_{(\alpha, q)}$.
	
	\begin{definition} \label{qha-slant}
		\cite{OnderSlant} A ruled surface $\varphi_{(\alpha, q)}$ is called $q$-slant or $a$-slant (resp. $h$-slant) ruled surface if its ruling $\vec{q}$  (resp. central normal $\vec{h}$) always makes a constant angle with a fixed direction.
	\end{definition}
	
	The first fundamental form $I$ and second fundamental form $II$ of $\varphi_{(\alpha, q)}$ are defined by
	\begin{equation*}
		I  = E d{{s}^{2}} + 2F ds du + G d{{u}^{2}},
		\hspace{8pt}
		II = L d{{s}^{2}} + 2M ds du + N d{{u}^{2}},
	\end{equation*}
	respectively, where
	\begin{equation} \label{EFGformulas}
		E=\left\langle \frac{\partial {{{\vec{\varphi }}}_{(\alpha ,q)}}}{\partial s},\frac{\partial {{{\vec{\varphi }}}_{(\alpha ,q)}}}{\partial s} \right\rangle,
		\hspace{6pt}
		F=\left\langle \frac{\partial {{{\vec{\varphi }}}_{(\alpha ,q)}}}{\partial s},\frac{\partial {{{\vec{\varphi }}}_{(\alpha ,q)}}}{\partial u} \right\rangle,
		\hspace{6pt}
		G=\left\langle \frac{\partial {{{\vec{\varphi }}}_{(\alpha ,q)}}}{\partial u},\frac{\partial {{{\vec{\varphi }}}_{(\alpha ,q)}}}{\partial u} \right\rangle,
	\end{equation}
	\begin{equation} \label{LMNformulas}
		L=\left\langle \frac{{{\partial }^{2}}{{{\vec{\varphi }}}_{(\alpha ,q)}}}{\partial {{s}^{2}}},\vec{U} \right\rangle,
		\hspace{6pt}
		M=\left\langle \frac{{{\partial }^{2}}{{{\vec{\varphi }}}_{(\alpha ,q)}}}{\partial s\partial u},\vec{U} \right\rangle,
		\hspace{6pt}
		N=\left\langle \frac{{{\partial }^{2}}{{{\vec{\varphi }}}_{(\alpha ,q)}}}{\partial {{u}^{2}}},\vec{U} \right\rangle .
	\end{equation}
	The Gaussian curvature $K$ and the mean curvature $H$ are defined by
	\begin{equation} \label{GaussianCurvatureFormula}
		K=\frac{LN-{{M}^{2}}}{EG-{{F}^{2}}},
	\end{equation}
	\begin{equation} \label{MeanCurvatureFormula}
		H=\frac{EN-2FM+GL}{2(EG-{{F}^{2}})}.
	\end{equation}
	respectively. An arbitrary surface is called minimal if $H=0$ at all points of the surface. Furthermore, a ruled surface is developable (or flat) if and only if $K=0$ \cite{doCarmo}.
	
	\begin{theorem}
		(Catalan Theorem) \cite{FomenkoTuzhilin} Among all ruled surfaces except planes only the helicoid and fragments of it are minimal.
	\end{theorem}
	
	\section{Osculating-type Ruled Surfaces}
	In this section, we define the osculating-type ruled surface of a curve $\alpha$ such that the ruling of the surface always lies in the osculating plane of $\alpha$ and also $\alpha$ is the base curve of the surface. Such a surface is defined as follows:
	
	\begin{definition} \label{OT-ruledSurfaceDefinition}
		Let $\alpha : I \subset \mathbb{R} \rightarrow {{\mathbb{R}}^{3}}$ be a smooth curve in the Euclidean 3-space $\mathbb{E}^3$ with arc-length parameter $s$, curvature $\kappa (s)$, torsion $\tau (s)$ and Frenet frame $\left\{ \vec{T}(s),\vec{N}(s),\vec{B}(s) \right\}$. Then, the ruled surface ${{\varphi }_{(\alpha ,{{q}_{o}})}}:I\times \mathbb{R}\to {{\mathbb{R}}^{3}}$ given by the parametric form
		\begin{equation} \label{OT-ruledSurfaceEquation}
			{{\vec{\varphi }}_{(\alpha ,{{q}_{o}})}}(s,u)=\vec{\alpha }(s)+u{{\vec{q}}_{o}}(s),
			\hspace{8pt}
			{{\vec{q}}_{o}}(s)=\cos \theta \vec{T}(s)+\sin \theta \vec{N}(s)
		\end{equation}
		is called the osculating-type (OT)-ruled surface of $\alpha$ where $\theta =\theta (s)$ is ${{C}^{\infty}}$-scalar angle function of arc-length parameter $s$ between unit vectors $\vec{q}_o$ and $\vec{T}$. Here, we use the index “$o$” to emphasize that the ruling always lies on the osculating plane $sp\left\{ \vec{T},\vec{N} \right\}$ of base curve $\alpha$.
	\end{definition}
	
	As we see from equation (\ref{OT-ruledSurfaceEquation}), when $\theta (s)=k\pi, (k \in \mathbb{Z})$, for all $s \in I$, the ruling becomes ${{\vec{q}}_{o}}=\pm \vec{T}$ and the OT-ruled surface $\varphi_{(\alpha, q_o)}$ becomes the developable tangent surface $\varphi_{(\alpha, T)}$ of $\alpha$. Similarly, when $\theta (s)=\pi /2+k\pi, (k \in \mathbb{Z})$ for all $s \in I$, the ruling becomes ${{\vec{q}}_{o}} = \pm \vec{N}$ and the OT-ruled surface $\varphi_{(\alpha, q_o)}$ becomes the principal normal surface $\varphi_{(\alpha, N)}$ of $\alpha$.
	
	\begin{remark} \label{remark}
		If $\alpha$ is a straight line, then $\varphi_{(\alpha, T)}$ is not a surface, only a line. So, for the case ${{\varphi }_{(\alpha ,{{q}_{o}})}}={{\varphi }_{(\alpha ,T)}}$, we always assume that $\alpha$ is not a straight line, i.e., $\kappa \ne 0$.
	\end{remark}
	
	Considering (\ref{OT-ruledSurfaceEquation}) and the fact that the binormal vector $\vec{B}$ of $\alpha$ is perpendicular to $sp\left\{ \vec{T},\vec{N} \right\}$, we get $\left\langle {{{\vec{q}}}_{o}},\vec{B} \right\rangle =0$. Therefore, we can define a unit vector $\vec{r} (s)$ as follows,
	\begin{equation} \label{r-vector}
		\vec{r}={{\vec{q}}_{o}}\times \vec{B}=\sin \theta \,\vec{T}-\cos \theta \vec{N}.
	\end{equation}
	Then, the frame $\left\{ {{{\vec{q}}}_{o}},\vec{B},\vec{r} \right\}$ is an orthonormal moving frame along $\alpha$ on the OT-ruled surface $\varphi_{(\alpha, q_o)}$. From equations (\ref{OT-ruledSurfaceEquation}) and (\ref{r-vector}), the relations between that frame and Frenet frame of $\alpha$ are given by $\vec{T}=\cos \theta \,{{\vec{q}}_{o}}+\sin \theta \,\vec{r}$ and $\vec{N}=\sin \theta \,{{\vec{q}}_{o}}-\cos \theta \,\vec{r}$. After some computations, for the derivative formulae of new frame $\left\{ {{{\vec{q}}}_{o}},\vec{B},\vec{r} \right\}$, we get
	\begin{equation*}
		\begin{bmatrix}
			{{{{\vec{q}}\hspace{2pt}'_{\hspace{-2pt}o}}}} \\
			{{\vec{B}}'} \\
			{{\vec{r}}\hspace{2pt}'} \\
		\end{bmatrix}
		=
		\begin{bmatrix}
			0 & \mu  & -\eta \\
			-\mu  & 0 & \xi \\
			\eta  & -\xi  & 0 \\
		\end{bmatrix}
		\begin{bmatrix}
			{{{\vec{q}}}_{o}} \\
			{\vec{B}} \\
			{\vec{r}} \\
		\end{bmatrix}
	\end{equation*}
	where $\eta (s)={\theta }'+\kappa $, $\mu (s)=\tau \sin \theta $, $\xi (s)=\tau \cos \theta $ are called the curvatures of OT-ruled surface $\varphi_{(\alpha, q_o)}$ according to the frame $\left\{ {{{\vec{q}}}_{o}},\vec{B},\vec{r} \right\}$. Then, the relationships between the curvatures $\kappa$, $\tau$ of base curve $\alpha$ and the curvatures $\eta$, $\mu$, $\xi$ of OT-ruled surface of OT-ruled surface $\varphi_{(\alpha, q_o)}$ are obtained as $\kappa =\eta -{\theta }'$, $\tau =\pm \sqrt{{{\mu }^{2}}+{{\xi }^{2}}}$. Now, using these relationships and considering the characterizations for general helix and slant helix, the following theorem is obtained:
	
	\begin{theorem} \label{alphahelixslanthelix}
		For the OT-ruled surface $\varphi_{(\alpha, q_o)}$, we have that
		\begin{enumerate} [(i)]
			\item $\alpha$ is a plane curve if and only if both $\mu$ and $\xi$ vanish.
			\item $\alpha$ is a general helix if and only if the function $\rho (s)=\pm \frac{\sqrt{{{\mu }^{2}}+{{\xi }^{2}}}}{\eta - \theta'}$ is constant.
			\item $\alpha$ is a slant helix if and only if the function
			\begin{equation*}
				\sigma (s)=\pm \frac{(\mu {\mu }'+\xi {\xi }')(\eta -{\theta }')-({{\mu }^{2}}+{{\xi }^{2}})({\eta }'-{\theta }'')}{\left[ {{(\eta -{\theta }')}^{2}}+{{\mu }^{2}}+{{\xi }^{2}} \right]{{\left( {{\mu }^{2}}+{{\xi }^{2}} \right)}^{1/2}}}
			\end{equation*}
			is constant.
		\end{enumerate}
	\end{theorem}
	
	Let now consider the special case that the base curve $\alpha$ is a plane curve, i.e., $\tau = 0$. Then, $\alpha$ lies on the osculating plane $sp\left\{ \vec{T},\vec{N} \right\}$ and has constant binormal vector $\vec{B}$. Since, the unit surface normal $\vec{U}$ of OT-ruled surface $\varphi_{(\alpha, q_o)}$ is always perpendicular to both $\vec{q}_o$ and $\vec{T}$, we have that $\vec{U}=\pm \vec{B}$. Then, the OT-ruled surface has a constant unit normal, that is, it is a plane. Conversely, if the OT-ruled surface $\varphi_{(\alpha, q_o)}$ is a plane with constant unit normal $\vec{U}$, since $\vec{U}\bot sp\left\{ {{{\vec{q}}}_{o}},\vec{T} \right\}$, from (\ref{OT-ruledSurfaceEquation}) we get $\vec{U}\bot sp\left\{ \vec{T},\vec{N} \right\}$ which gives $\vec{U}=\pm \vec{B}$ is a constant vector. Then, $\tau = 0$, i.e., $\alpha$ is a plane curve and we have the followings:
	
	\begin{theorem} \label{PlaneOT}
		The OT-ruled surface $\varphi_{(\alpha, q_o)}$ is a plane if and only if the base curve $\alpha$ is a plane curve.
	\end{theorem}
	
	Clearly, Theorem 3.4 gives the following corollary:
	
	\begin{theorem} \label{notTN}
		If ${{\varphi}_{(\alpha ,{{q}_{o}})}}\ne {{\varphi}_{(\alpha ,T)}}$ and ${{\varphi}_{(\alpha ,{{q}_{o}})}}\ne {{\varphi}_{(\alpha ,N)}}$,  the followings are equivalent
			\begin{equation*}
				\begin{split}
					& (i) \hspace{4pt} \alpha \hspace{4pt} \textit{is a plane curve}. \hspace{16pt} (ii) \hspace{4pt} \textit{The OT-ruled surface} \hspace{4pt} \varphi_{(\alpha, q_o)} \hspace{4pt} \textit{is a plane}. \\
					& (iii) \hspace{4pt} \mu=0. \hspace{65pt} (iv) \hspace{4pt} \xi=0.
				\end{split}
			\end{equation*}
	\end{theorem}
	
	Now, we will give other characterizations and geometric properties of the OT-ruled surfaces.
	
	\begin{theorem}
		The set of the singular points of OT-ruled surface $\varphi_{(\alpha, q_o)}$ is given by
		\begin{equation*}
			S = \left\{ ({{s}_{0}},{{u}_{0}}) \in I \times \mathbb{R}: \theta ({{s}_{0}})=k\pi, {{u}_{0}}=0, k\in \mathbb{Z} \right\}.
		\end{equation*}
	\end{theorem}
	\begin{proof}
		From the partial derivatives of ${{\vec{\varphi }}_{(\alpha, {{q}_{o}})}}(s,u)=\vec{\alpha }(s)+u{{\vec{q}}_{o}}(s)$, we get
		\begin{equation} \label{partialderivativesofvarphi}
			\frac{\partial {{{\vec{\varphi }}}_{(\alpha ,{{q}_{o}})}}}{\partial s}=\cos \theta \,{{\vec{q}}_{o}}+u\mu \vec{B}+(\sin \theta -u\eta )\vec{r}, \hspace{8pt} \frac{\partial {{{\vec{\varphi }}}_{(\alpha ,{{q}_{o}})}}}{\partial u}={{\vec{q}}_{o}}.
		\end{equation}
		Therefore, the direction of surface normal is given by the vector
		\begin{equation*}
			\frac{\partial {{{\vec{\varphi }}}_{(\alpha ,{{q}_{o}})}}}{\partial s}\times \frac{\partial {{{\vec{\varphi }}}_{(\alpha ,{{q}_{o}})}}}{\partial u}=(\sin \theta -u\eta )\vec{B}-u\mu  \vec{r}.
		\end{equation*}
		Then, the OT-ruled surface $\varphi_{(\alpha, q_o)}$ has singular points if and only if the system
		\begin{equation} \label{SingularSystem}
			\begin{cases}
				\sin \theta -u\eta = 0,\\
				u \mu = 0
			\end{cases}
		\end{equation}
		holds. Let now assume that $u=0$. Then, from the first equality, it follows $\theta ({{s}_{0}})=k\pi, (k \in \mathbb{Z}, {{s}_{0}} \in I)$. When this satisfies for all $s \in I$, we have ${{\varphi }_{(\alpha, {{q}_{o}})}}={{\varphi }_{(\alpha ,T)}}$ and the locus of the singular points is the base curve $\alpha$. If $u\ne 0$, from the system (\ref{SingularSystem}), we get $u(s)=\frac{\sin \theta }{\eta }$ and $\mu =0$. Since we assume that singular points exist, from Theorem \ref{PlaneOT}, we have $\tau \ne 0$. Otherwise, the surface is a plane and regular. Then, $\mu=0$ implies that $\sin \theta =0$ which is a contradiction with the assumption that $u \ne 0$. And so, the system (\ref{SingularSystem}) only holds if and only if $u=0$, $\theta ({{s}_{0}})=k\pi, (k\in \mathbb{Z}, {{s}_{0}}\in I)$.
	\end{proof}
	
	Hereafter, for the sake of simplicity, we will take $f=\sin \theta -u\eta $ and $g=u\mu $.
	
	\begin{proposition} \label{developable}
		The OT-ruled surface $\varphi_{(\alpha, q_o)}$ is developable if and only if $\varphi_{(\alpha, q_o)}$ is a plane or ${{\varphi }_{(\alpha, {{q}_{o}})}}={{\varphi }_{(\alpha ,T)}}$.
	\end{proposition}
	\begin{proof}
		For the surface $\varphi_{(\alpha, q_o)}$, we have $\det (\vec{\alpha }', \vec{q}_o, \vec{q}\hspace{2pt}'_{\hspace{-2pt} o}) = \mu \sin \theta$. Considering Theorem \ref{PlaneOT}, we have the desired result.
	\end{proof}
	
	\begin{proposition} \label{cylindrical}
		Among all OT-ruled surfaces $\varphi_{(\alpha, q_o)}$, only the plane is cylindrical.
	\end{proposition}
	\begin{proof}
		Since a ruled surface is called cylindrical if and only if the direction of the ruling is a constant vector, we get $\vec{q}\hspace{2pt}'_{\hspace{-2pt} o} = 0$ if and only if
		\begin{equation} \label{diffq_o}
			-\eta \sin \theta \vec{T} + \eta \cos \theta \vec{N} + \tau \sin \theta \vec{B} = 0.
		\end{equation}
		If ${{\varphi }_{(\alpha, {{q}_{o}})}}={{\varphi }_{(\alpha, T)}}$, then $\theta (s)=k \pi$ for all $s\in I$ and (\ref{diffq_o}) gives $\eta =0$, which implies that $\kappa =0$, which is a contradiction with Remark \ref{remark}. If ${{\varphi }_{(\alpha ,{{q}_{o}})}}\ne {{\varphi }_{(\alpha ,T)}}$, then from (\ref{diffq_o}) we have $\tau =0$, $\eta =0$ which gives $\theta (s)=-\int_{0}^{s}{\kappa (s)ds}$ and Theorem \ref{PlaneOT} gives that $\varphi_{(\alpha, q_o)}$ is a plane.
	\end{proof}
	
	Proposition \ref{developable} and Proposition \ref{cylindrical} give the following corollary:
	
	\begin{corollary}
		If the OT-ruled surface $\varphi_{(\alpha, q_o)}$ is cylindrical, then it is a plane with the parametric form
		\begin{equation*}
			{{\vec{\varphi }}_{(\alpha ,{{q}_{o}})}}(s,u)=\vec{\alpha }(s)+u\left( \cos \left( \int_{0}^{s}{\kappa (s)ds} \right)\vec{T}(s)-\sin \left( \int_{0}^{s}{\kappa (s)ds} \right)\vec{N}(s) \right)
		\end{equation*}
	\end{corollary}
	
	\begin{proposition} \label{strictionlineprop}
		The base curve $\alpha$ of the OT-ruled surface $\varphi_{(\alpha, q_o)}$ is also its striction line if and only if $\theta (s)=-\int_{0}^{s}{\kappa (s)ds}$ or ${{\varphi }_{(\alpha ,{{q}_{o}})}}={{\varphi }_{(\alpha ,T)}}$.
	\end{proposition}
	\begin{proof}
		The base curve $\alpha$ is the striction line of $\varphi_{(\alpha, q_o)}$ if and only if $\langle \vec{\alpha}', \vec{q}\hspace{2pt}'_{\hspace{-2pt} o} \rangle = 0$. Therefore, we get $\langle \vec{\alpha}', \vec{q}\hspace{2pt}'_{\hspace{-2pt} o} \rangle = - \eta \sin \theta$ which gives the desired result.
	\end{proof}
	
	From Proposition \ref{strictionlineprop}, it is clear that the set of the intersection points of base curve $\alpha$ and striction curve $c$ is $V = S \cup Y$, where $S$ is the set of singular points of $\varphi_{(\alpha, q_o)}$ and
	\begin{equation*}
		Y=\left\{ ({{s}_{0}},{{u}_{0}}) \in I \times \mathbb{R}: {\theta }'({{s}_{0}})=-\kappa ({{s}_{0}}), {{u}_{0}}=0 \right\}.
	\end{equation*}
	It is clear that the points of $Y$ are non-singular.
	
	Let now investigate the special curves lying on the OT-surface $\varphi_{(\alpha, q_o)}$. The Gauss map (or the unit surface normal) of the OT-ruled surface $\varphi_{(\alpha, q_o)}$ is given by
	\begin{equation} \label{GaussMap}
		\vec{U}(s,u)=\frac{1}{\sqrt{{{f}^{2}}+{{g}^{2}}}}\left( f\vec{B}-g\,\vec{r} \right).
	\end{equation}
	Then, for the base curve $\alpha$ we have the followings:
	
	\begin{theorem}
		The base curve $\alpha$ is a geodesic on the OT-ruled surface $\varphi_{(\alpha, q_o)}$ if and only if $\alpha$ is a straight line.
	\end{theorem}
	\begin{proof}
		We know that $\alpha$ is a geodesic on $\varphi_{(\alpha, q_o)}$ if and only if the condition
		\begin{equation} \label{geodesiccondition}
			\vec{U} \times {\vec{\alpha }}'' = 0
		\end{equation}
		satisfies. Then, by using (\ref{GaussMap}), from (\ref{geodesiccondition}) we get
		\begin{equation}
			\vec{U}\times {\vec{\alpha }}''=\frac{1}{\sqrt{{{f}^{2}}+{{g}^{2}}}}\left( -\kappa f\,\vec{T}-g\kappa \sin \theta \,\vec{B} \right)
		\end{equation}
		and that $\alpha$ is a geodesic curve on $\varphi_{(\alpha, q_o)}$ if and only if the system
		\begin{equation*}
			\begin{cases}
				\kappa f = 0\\
				g \kappa \sin \theta = 0
			\end{cases}
		\end{equation*}
		holds. If we assume ${{\varphi }_{(\alpha ,{{q}_{o}})}}\ne {{\varphi }_{(\alpha ,T)}}$, from the last system it follows
		\begin{equation*}
			\kappa f=0,\hspace{4pt} g\kappa =0,
		\end{equation*}
		which gives that $\kappa=0$, i.e., $\alpha$ is a straight line or the system
		\begin{equation*}
			f=0, \hspace{4pt} g=0,
		\end{equation*}
		holds. But for the last system, considering (\ref{SingularSystem}), it follows that the system has a solution as a curve if and only if ${{\varphi }_{(\alpha ,{{q}_{o}})}}={{\varphi }_{(\alpha ,T)}}$ which is a contradiction by the assumption ${{\varphi }_{(\alpha ,{{q}_{o}})}}\ne {{\varphi }_{(\alpha ,T)}}$ and so, we eliminate this case. If ${{\varphi }_{(\alpha ,{{q}_{o}})}}={{\varphi }_{(\alpha ,T)}}$, considering Remark \ref{remark}, we should take $\kappa \ne 0$. But for this case, the system gives that $\eta =\kappa =0$, which is a contradiction. Then we have that $\alpha$ is a geodesic on $\varphi_{(\alpha, q_o)}$ if and only if $\alpha$ is a straight line.
	\end{proof}
	
	\begin{theorem}
		Let $\alpha$ have non-vanishing curvature $\kappa$. Then, $\alpha$ is an asymptotic curve on the OT-ruled surface $\varphi_{(\alpha, q_o)}$ if and only if one of the followings hold:
		\begin{equation*}
			(i) \hspace{4pt} \varphi_{(\alpha, q_o)} \hspace{4pt} \textit{is a plane} \hspace{12pt} (ii) \hspace{4pt} {{\varphi }_{(\alpha ,{{q}_{o}})}}={{\varphi }_{(\alpha ,T)}} \hspace{12pt} (iii) \hspace{4pt} {{\varphi }_{(\alpha ,{{q}_{o}})}}={{\varphi }_{(\alpha ,N)}}.
		\end{equation*}
	\end{theorem}
	\begin{proof}
		 $\alpha$ is an asymptotic curve on $\varphi_{(\alpha, q_o)}$ if and only if $\langle \vec{U},{\vec{\alpha }}'' \rangle =0.$ Then, we get
		 \begin{equation} \label{asymptoticcondition}
			 \left\langle \vec{U},{\vec{\alpha }}'' \right\rangle =\frac{u\kappa \tau \cos \theta \sin \theta }{\sqrt{{{f}^{2}}+{{g}^{2}}}}
		 \end{equation}
		 From (\ref{asymptoticcondition}), we obtain that $\langle \vec{U},{\vec{\alpha }}'' \rangle = 0$ if and only if $\tau=0$ or $\sin \theta = 0$ or $\cos \theta = 0$.
	\end{proof}
	
	\begin{theorem}
		The base curve $\alpha$ is a line of curvature on the OT-ruled surface $\varphi_{(\alpha, q_o)}$ if and only if $\varphi_{(\alpha, q_o)}$ is a plane.
	\end{theorem}
	\begin{proof}
		The curve $\alpha$ is a line of curvature on the OT-ruled surface $\varphi_{(\alpha, q_o)}$ if and only if $\vec{U}_\alpha' \times \vec{\alpha}' = 0$ holds where $\vec{U}_\alpha$ is the unit surface normal along the curve $\alpha$ and for which we have ${{\vec{U}}_{\alpha }}=\vec{B}$. Then, it follows
		\begin{equation} \label{lineofcurvaturecondition}
			{{{\vec{U}}'}_{\alpha }}\times {\vec{\alpha }}'=-\tau \vec{B}
		\end{equation}
		The equation (\ref{lineofcurvaturecondition}) is equal to zero if and only if $\tau = 0$ and from Theorem \ref{PlaneOT}, we have that $\varphi_{(\alpha, q_o)}$ is a plane.
	\end{proof}
	
	Now, let us examine first and second fundamental coefficients of the OT-ruled surface $\varphi_{(\alpha, q_o)}$. From (\ref{EFGformulas}) and (\ref{LMNformulas}), we get
	\begin{equation} \label{EFG}
		E={{f}^{2}}+{{g}^{2}}+{{\cos }^{2}}\theta, \hspace{4pt} F=\cos \theta, \hspace{4pt} G=1
	\end{equation}
	\begin{equation} \label{LMN}
		L=\frac{-({{f}^{2}}+{{g}^{2}})\xi +\mu \sin \theta \cos \theta -{{g}^{2}}{{\left( \frac{f}{g} \right)}_{\hspace{-2pt}s}} }{\sqrt{{{f}^{2}}+{{g}^{2}}}}, \hspace{4pt} M=\frac{\mu \sin \theta }{\sqrt{{{f}^{2}}+{{g}^{2}}}}, \hspace{4pt} N=0
	\end{equation}
	By using the fundamental coefficients computed in (\ref{EFG}) and (\ref{LMN}), from (\ref{GaussianCurvatureFormula}) and (\ref{MeanCurvatureFormula}) the Gaussian curvature $K$ and the mean curvature $H$ of OT-ruled surface $\varphi_{(\alpha, q_o)}$ are obtained as
	\begin{equation} \label{KandH}
		K=-\frac{{{\mu }^{2}}{{\sin }^{2}}\theta }{{{\left( {{f}^{2}}+{{g}^{2}} \right)}^{2}}}, \hspace{4pt} H=-\frac{({{f}^{2}}+{{g}^{2}})\xi +\mu \sin \theta \cos \theta +{{g}^{2}}{{\left( \frac{f}{g} \right)}_{\hspace{-2pt}s}}}{2{{\left( {{f}^{2}}+{{g}^{2}} \right)}^{3/2}}}
	\end{equation}
	respectively. From (\ref{KandH}), it follows that $K=0$ if and only if $\tau = 0$ or $\sin \theta = 0$. This result coincides with Proposition \ref{developable}.
	
	It is clear that if $\tau = 0$, then $H=0$ and the OT-ruled surface $\varphi_{(\alpha, q_o)}$ is minimal. If $\tau \ne 0$ and $\sin \theta =0$, then from (\ref{KandH}) we get $H=\frac{\tau }{2u\eta }\ne 0$ Therefore, in this case, the tangent surface ${{\varphi }_{(\alpha ,T)}}$ cannot be minimal. Then, followings are obtained:
	
	\begin{theorem}
		\begin{enumerate} [(i)]
			\item The OT-ruled surface $\varphi_{(\alpha, q_o)}$ is minimal if and only if the equality
			\begin{equation*}
				({{f}^{2}}+{{g}^{2}})\xi +\mu \sin \theta \cos \theta +{{g}^{2}}{{\left( \frac{f}{g} \right)}_{\hspace{-2pt}s}}=0
			\end{equation*}
			satisfies.
			\item If $\tau \ne 0$, there is no minimal tangent surface ${{\varphi }_{(\alpha ,T)}}$.
			\item The principal normal surface ${{\varphi }_{(\alpha ,N)}}$ is minimal if and only if $fu{\mu }'-g{{f}_{s}}=0$, where ${{f}_{s}}=\partial f/\partial s$.
		\end{enumerate}
	\end{theorem}
	
	Furthermore, considering Catalan Theorem, we have the following corollary:
	
	\begin{corollary}
		If the base curve $\alpha$ is not a plane curve, the OT-ruled surface $\varphi_{(\alpha, q_o)}$ is a helicoid if and only if $({{f}^{2}}+{{g}^{2}})\xi +\mu \sin \theta \cos \theta +{{g}^{2}}{{\left( \frac{f}{g} \right)}_{\hspace{-2pt}s}}=0$ holds.
	\end{corollary}
	
	Now, we will consider the special curves lying on an OT-ruled surface $\varphi_{(\alpha, q_o)}$.
	
	Let us consider the tangent space ${{T}_{p}}{{\varphi }_{(\alpha ,{{q}_{o}})}}$ and its base $\left\{ \frac{\partial {{{\vec{\varphi }}}_{(\alpha ,{{q}_{o}})}}}{\partial s},\frac{\partial {{{\vec{\varphi }}}_{(\alpha ,{{q}_{o}})}}}{\partial u} \right\}$ at a point $p\in {{\varphi }_{(\alpha ,{{q}_{o}})}}$. For any tangent vector ${{\vec{v}}_{p}}\in {{T}_{p}}{{\varphi }_{(\alpha ,{{q}_{o}})}}$, the Weingarten map of the OT-ruled surface $\varphi_{(\alpha, q_o)}$ is defined by ${{S}_{p}}=-{{D}_{p}}\vec{v}:{{T}_{p}}{{\varphi }_{(\alpha ,{{q}_{o}})}}\to {{T}_{{{{\vec{v}}}_{p}}}}{{S}^{2}}$ where ${{S}^{2}}$ is unit sphere with center origin. Therefore, we have
	\begin{equation*}
		\begin{split}
			S_p \left( \frac{\partial {{{\vec{\varphi }}}_{(\alpha ,{{q}_{o}})}}}{\partial s} \right) & = - D_{\frac{\partial {{{\vec{\varphi }}}_{(\alpha ,{{q}_{o}})}}}{\partial s}} \vec{U}(s,u),\\
			& = A_1 (s,u) \frac{\partial {{{\vec{\varphi }}}_{(\alpha ,{{q}_{o}})}}}{\partial s} + A_2 (s,u) \frac{\partial {{{\vec{\varphi }}}_{(\alpha ,{{q}_{o}})}}}{\partial u},
		\end{split}
	\end{equation*}
	and
	\begin{equation*}
		\begin{split}
			S_p \left( \frac{\partial {{{\vec{\varphi }}}_{(\alpha ,{{q}_{o}})}}}{\partial u} \right) & = - D_{\frac{\partial {{{\vec{\varphi }}}_{(\alpha ,{{q}_{o}})}}}{\partial u}} \vec{U}(s,u),\\
			& = B_1 (s,u) \frac{\partial {{{\vec{\varphi }}}_{(\alpha ,{{q}_{o}})}}}{\partial s} + B_2 (s,u) \frac{\partial {{{\vec{\varphi }}}_{(\alpha ,{{q}_{o}})}}}{\partial u},
		\end{split}
	\end{equation*}
	where
	\begin{equation*}
		\begin{split}
			{{A}_{1}}&=\frac{-1}{{{\left( {{f}^{2}}+{{g}^{2}} \right)}^{3/2}}}\left[ {{g}^{2}}{{\left( \frac{f}{g} \right)}_{\hspace{-2pt}s}}+\left( {{f}^{2}}+{{g}^{2}} \right)\xi  \right],\\
			{{A}_{2}}&=\frac{1}{{{\left( {{f}^{2}}+{{g}^{2}} \right)}^{3/2}}}\left[ \left( {{f}^{2}}+{{g}^{2}} \right)(f\mu +g\eta +\xi \cos \theta )+{{g}^{2}}\cos \theta {{\left( \frac{f}{g} \right)}_{\hspace{-2pt}s}} \right],\\
			{{B}_{1}}&=\frac{\mu \sin \theta }{{{\left( {{f}^{2}}+{{g}^{2}} \right)}^{3/2}}}, \hspace{4pt} {{B}_{2}}=\frac{-\mu \cos \theta \sin \theta }{{{\left( {{f}^{2}}+{{g}^{2}} \right)}^{3/2}}}.
		\end{split}
	\end{equation*}
	
	Thus, the matrix form of the Weingarten map can be given by
	\begin{equation} \label{WeingartenMap}
		{S}_{p}
		=
		\begin{bmatrix}
			{{A}_{1}} & {{B}_{1}}\\
			{{A}_{2}} & {{B}_{2}}\\
		\end{bmatrix}
	\end{equation}
	From (\ref{WeingartenMap}), one can easily compute the Gaussian curvature $K$ and the mean curvature $H$ by considering the equalities $K=\det ({{S}_{p}})$ and $H=\frac{1}{2}tr({{S}_{p}})$ and the results given in (\ref{KandH}) are obtained. Moreover, from these results, for the parameter curves, we have the following corollary:
	
	\begin{corollary}
		\begin{enumerate} [(i)]
			\item The parameter curves ${{\vec{\varphi }}_{(\alpha, {{q}_{o}})}}(s,{{u}_{0}})$, ($u_0$ is constant) are lines of curvature if and only if ${{A}_{2}}=0$ or equivalently, $\left( {{f}^{2}}+{{g}^{2}} \right)(f\mu +g\eta +\xi \cos \theta )+{{g}^{2}}\cos \theta {{\left( \frac{f}{g} \right)}_{\hspace{-2pt}s}}=0$ holds.
			\item The parameter curves ${{\vec{\varphi }}_{(\alpha ,{{q}_{o}})}}({{s}_{0}},u)$, ($s_0$ is constant) are lines of curvature if and only if ${{B}_{1}}=0$ or equivalently, the OT-ruled surface $\varphi_{(\alpha, q_o)}$ is a plane or ${{\varphi }_{(\alpha ,{{q}_{o}})}}={{\varphi }_{(\alpha, T)}}$.
		\end{enumerate}
	\end{corollary}
	
	Considering the characteristic equation $\det \left( {{S}_{p}}-\lambda \text{I} \right)=0$, the principal curvatures of OT-ruled surface $\varphi_{(\alpha, q_o)}$ are obtained as
	\begin{equation*}
		{{\lambda }_{1,2}}=\frac{{{A}_{1}}+{{B}_{2}}\pm \sqrt{{{\left( {{A}_{1}}-{{B}_{2}} \right)}^{2}}+4{{A}_{2}}{{B}_{1}}}}{2}
	\end{equation*}
	where $I$ is $2 \times 2$ unit matrix. Then, the principal directions are obtained as
	\begin{equation*}
		{{\vec{e}}_{1}}=\frac{1}{{{B}_{1}}}\left( {{B}_{1}}\frac{\partial {{{\vec{\varphi }}}_{(\alpha ,{{q}_{o}})}}}{\partial s}+k{{A}_{2}}\frac{\partial {{{\vec{\varphi }}}_{(\alpha ,{{q}_{o}})}}}{\partial u} \right), \hspace{4pt} {{\vec{e}}_{2}}=\frac{1}{m{{A}_{2}}}\left( {{B}_{1}}\frac{\partial {{{\vec{\varphi }}}_{(\alpha ,{{q}_{o}})}}}{\partial s}+m{{A}_{2}}\frac{\partial {{{\vec{\varphi }}}_{(\alpha ,{{q}_{o}})}}}{\partial u} \right),
	\end{equation*}
	where $k,m$ are scalar functions such that
	\begin{equation*}
		\frac{{{\lambda }_{1}}-{{A}_{1}}}{{{A}_{2}}}=\frac{{{B}_{1}}}{{{\lambda }_{1}}-{{A}_{2}}}=k,  \hspace{4pt}   \frac{{{\lambda }_{2}}-{{A}_{1}}}{{{A}_{2}}}=\frac{{{B}_{1}}}{{{\lambda }_{2}}-{{A}_{2}}}=m.
	\end{equation*}
	
	Let now $\beta (t)={{\varphi }_{(\alpha ,{{q}_{o}})}}\left( s(t),u(t) \right)$ be a unit speed curve on $\varphi_{(\alpha, q_o)}$ with arc length parameter $t$ and unit tangent vector ${{\vec{v}}_{p}}\in {{T}_{p}}{{\varphi }_{(\alpha ,{{q}_{o}})}}$ at the point $\beta ({{t}_{o}})=p$ on $\varphi_{(\alpha, q_o)}$. The derivative of $\beta$ with respect to $t$ has the form
	\begin{equation*}
		\dot{\vec{\beta }}(t)=\frac{\partial {{{\vec{\varphi }}}_{(\alpha ,{{q}_{o}})}}}{\partial s}\frac{ds}{dt}+\frac{\partial {{{\vec{\varphi }}}_{(\alpha ,{{q}_{o}})}}}{\partial u}\frac{du}{dt}
	\end{equation*}
	where $\dot{\vec{\beta }}=\frac{d\vec{\beta }}{dt}$. For this tangent vector, we can write
	\begin{equation} \label{TangentVector}
		{{\vec{v}}_{p}}=C(s,u)\frac{\partial {{{\vec{\varphi }}}_{(\alpha ,{{q}_{o}})}}}{\partial s}+D(s,u)\frac{\partial {{{\vec{\varphi }}}_{(\alpha ,{{q}_{o}})}}}{\partial u}
	\end{equation}
	where $C,D$ are smooth functions defined by $C(t)=C\left( s(t),u(t) \right)=\frac{ds}{dt}=\dot{s}$ and $D(t)=D\left( s(t),u(t) \right)=\frac{du}{dt}=\dot{u}$. Substituting (\ref{partialderivativesofvarphi}) in (\ref{TangentVector}), gives
	\begin{equation*}
		{{\vec{v}}_{p}}=\left( C\cos \theta +D \right){{\vec{q}}_{o}}+Cg\vec{B}+Cf\vec{r}
	\end{equation*}
	where ${{\left( C\cos \theta +D \right)}^{2}}+{{C}^{2}}({{f}^{2}}+{{g}^{2}})=1$. Also, by using the linearity of the Weingarten map, we get
	\begin{equation*}
		{{S}_{p}}({{\vec{v}}_{p}})=\left[ \cos \theta \left( C{{A}_{1}}+D{{B}_{1}} \right)+\left( C{{A}_{2}}+D{{B}_{2}} \right) \right]{{\vec{q}}_{o}}+\left( C{{A}_{1}}+D{{B}_{1}} \right)\left( g\vec{B}+f\vec{r} \right)
	\end{equation*}
	and so on, the normal curvature ${{k}_{n}}$ in the direction ${{\vec{v}}_{p}}$ is computed as
	\begin{equation} \label{normalcurvature}
		\begin{split}
			{{k}_{n}}({{{\vec{v}}}_{p}})&=\left\langle {{S}_{p}}({{{\vec{v}}}_{p}}),{{{\vec{v}}}_{p}} \right\rangle\\
			&=C\left[ \left( C\cos \theta +D \right)\left( {{A}_{1}}\cos \theta +{{A}_{2}} \right)+\left( C{{A}_{1}}+D{{B}_{1}} \right)\left( {{f}^{2}}+{{g}^{2}} \right) \right].
		\end{split}
	\end{equation}
	
	Then, from (\ref{normalcurvature}), we have the following theorem:
	
	\begin{theorem}
		The surface curve $\beta (t)={{\varphi }_{(\alpha ,{{q}_{o}})}}\left( s(t),u(t) \right)$ with unit tangent ${{\vec{v}}_{p}}$ is an asymptotic curve if and only if $\beta (t)$ is a ruling or $\left( C\cos \theta +D \right)\left( {{A}_{1}}\cos \theta +{{A}_{2}} \right)+\left( C{{A}_{1}}+D{{B}_{1}} \right)\left( {{f}^{2}}+{{g}^{2}} \right)=0$ holds.
	\end{theorem}
	Similarly, the geodesic curvature ${{\kappa }_{g}}$ and the geodesic torsion ${{\tau }_{g}}$ of the curve $\beta (t)={{\varphi }_{(\alpha, {{q}_{o}})}}\left( s(t),u(t) \right)$ are computed as
	\begin{equation*}
		\begin{split}
			{{\kappa }_{g}}=\frac{1}{\sqrt{{{f}^{2}}+{{g}^{2}}}}&\left[ \left( C\cos \theta +D \right) \right.\left( -\left( {{f}^{2}}+{{g}^{2}} \right)\dot{C} \right. \\ 
			& \left. +C\left( \eta f-\mu g \right)\left( 2D+\cos \theta  \right)-\frac{1}{2}C{{\left( {{f}^{2}}+{{g}^{2}} \right)}_{s}} \right) \\ 
			& \,+Cg\left( \dot{C}g\cos \theta -Cg{\theta }'\sin \theta -\mu C{{g}^{2}}+fC\eta g+\dot{D}g \right) \\ 
			& \left. \,+Cf\left( \dot{C}f\cos \theta -Cf{\theta }'\sin \theta -\mu fgC+C{{f}^{2}}\eta +\dot{D}f \right) \right] \\ 
		\end{split}
	\end{equation*}
	and
	\begin{equation*}
		{{\tau }_{g}}=\sqrt{{{f}^{2}}+{{g}^{2}}}\left[ C\left( C{{A}_{2}}+D{{B}_{2}} \right)-D\left( C{{A}_{1}}+D{{B}_{1}} \right) \right]
	\end{equation*}
	respectively. Then, we have the followings:
	
	\begin{theorem}
		The surface curve $\beta (t)={{\varphi }_{(\alpha ,{{q}_{o}})}}\left( s(t),u(t) \right)$ with unit tangent ${{\vec{v}}_{p}}$ is a geodesic if and only if
		\begin{equation*}
			\begin{split}
			&\left( C\cos \theta +D \right)\left( -\left( {{f}^{2}}+{{g}^{2}} \right)\dot{C}+C\left( \eta f-\mu g \right)\left( 2D+\cos \theta  \right)-\frac{1}{2}C{{\left( {{f}^{2}}+{{g}^{2}} \right)}_{s}} \right) \\ 
			& \,\,\,\,\,\,\,+Cg\left( \dot{C}g\cos \theta -Cg{\theta }'\sin \theta -\mu C{{g}^{2}}+fC\eta g+\dot{D}g \right) \\ 
			& \,\,\,\,\,\,\,+Cf\left( \dot{C}f\cos \theta -Cf{\theta }'\sin \theta -\mu fgC+C{{f}^{2}}\eta +\dot{D}f \right)=0 \\ 
			\end{split}
		\end{equation*}
		holds.
	\end{theorem}
	
	Now, we can investigate some special cases:
	
	\textit{Case 1:} Let $\varphi_{(\alpha, q_o)}$ be $\varphi_{(\alpha, T)}$. Then,
	\begin{equation*}
		\begin{split}
			{{k}_{n}}&={{C}^{2}}u\kappa \tau \\
			{{\kappa }_{g}}&=C\left( C+D \right)\left[ u{\kappa }'+\kappa \left( 2D+1 \right) \right]+u\kappa \left( \dot{C}D-C\dot{D} \right)+{{C}^{2}}{{u}^{2}}{{\kappa }^{3}} \\
			{{\tau }_{g}}&=C\tau \left( C+D \right) 
		\end{split}
	\end{equation*}
	and for the curve $\beta (t)={{\varphi }_{(\alpha ,T)}}\left( s(t),u(t) \right)$, we have followings:
	\begin{enumerate}[(i)]
		\item $\beta (t)={{\varphi }_{(\alpha ,T)}}\left( s(t),u(t) \right)$ is an asymptotic curve if and only if $\beta (t)$ is a ruling or $\alpha$ is a plane curve.
		\item $\beta (t)={{\varphi }_{(\alpha ,T)}}\left( s(t),u(t) \right)$ is a geodesic if and only if
		\begin{equation*}
			C\left( C+D \right)\left[ u{\kappa }'+\kappa \left( 2D+1 \right) \right]+u\kappa \left( \dot{C}D-C\dot{D} \right)+{{C}^{2}}{{u}^{2}}{{\kappa }^{3}}=0
		\end{equation*}
		holds.
		\item $\beta (t)={{\varphi }_{(\alpha ,T)}}\left( s(t),u(t) \right)$ is a line of curvature if and only if one of the followings holds
		\subitem \hspace{-16pt} (a) $\beta (t)$ is a ruling, \hspace{4pt} (b) $\alpha$ is a plane curve, \hspace{4pt} (c) $s(t)=-u(t)+c$, where $c$ is integration constant.
	\end{enumerate}
	
	\textit{Case 2}: Let $\varphi_{(\alpha, q_o)}$ be $\varphi_{(\alpha, N)}$. Then,
	\begin{equation*}
		\begin{split}
			{{k}_{n}}&=\frac{C\left[ C\left( {{u}^{2}}{\kappa }'\tau +\left( 1-u\kappa  \right)u{\tau }' \right)+2D\tau  \right]}{\sqrt{{{\left( 1-u\kappa  \right)}^{2}}+{{u}^{2}}{{\tau }^{2}}}}\\
			{{\kappa }_{g}}&=\frac{1}{\sqrt{{{\left( 1-u\kappa  \right)}^{2}}+{{u}^{2}}{{\tau }^{2}}}}\left[ \dot{C}D \right.\left( -\left( {{\left( 1-u\kappa  \right)}^{2}}+{{u}^{2}}{{\tau }^{2}} \right) \right. \\ 
			& \hspace{50pt} \left. +2CD\left( \kappa -u\left( {{\kappa }^{2}}+{{\tau }^{2}} \right) \right)-C\left( \left( 1-u\kappa  \right){\kappa }'+{{u}^{2}}\tau {\tau }' \right) \right) \\ 
			& \hspace{50pt} +Cu\tau \left( -\tau C{{u}^{2}}{{\tau }^{2}}+C\left( 1-u\kappa  \right)u\kappa \tau +\dot{D}u\tau  \right) \\ 
			& \hspace{50pt} \left. +C\left( 1-u\kappa  \right)\left( -C\left( 1-u\kappa  \right)u{{\tau }^{2}}+C\kappa {{\left( 1-u\kappa  \right)}^{2}}+\dot{D}\left( 1-u\kappa  \right) \right) \right]\\
			{{\tau }_{g}}&={{C}^{2}}\tau -\frac{D\left[ Cu\left( u{\kappa }'\tau +\left( 1-u\kappa  \right){\tau }' \right)-D\tau  \right]}{\sqrt{{{\left( 1-u\kappa  \right)}^{2}}+{{u}^{2}}{{\tau }^{2}}}}
		\end{split}
	\end{equation*}
	and for the curve $\beta (t)={{\varphi }_{(\alpha ,N)}}\left( s(t),u(t) \right)$ with unit tangent ${{\vec{v}}_{p}}$, we have followings:
	\begin{enumerate} [(i)]
		\item $\beta (t)={{\varphi }_{(\alpha ,N)}}\left( s(t),u(t) \right)$ is an asymptotic curve if and only if $\beta (t)$ is a ruling or
		\begin{equation*}
			C\left( {{u}^{2}}{\kappa }'\tau +\left( 1-u\kappa  \right)u{\tau }' \right)+2D\tau =0
		\end{equation*}
		holds.
		\item $\beta (t)={{\varphi }_{(\alpha ,N)}}\left( s(t),u(t) \right)$ is a geodesic if and only if
		\begin{equation*}
			\begin{split}
				& \dot{C}D\left( -\left( {{\left( 1-u\kappa  \right)}^{2}}+{{u}^{2}}{{\tau }^{2}} \right) \right. \\ 
				& \hspace{30pt} \left. +2CD\left( \kappa -u\left( {{\kappa }^{2}}+{{\tau }^{2}} \right) \right)-C\left( \left( 1-u\kappa  \right){\kappa }'+{{u}^{2}}\tau {\tau }' \right) \right) \\ 
				& \hspace{30pt} +Cu\tau \left( -\tau C{{u}^{2}}{{\tau }^{2}}+C\left( 1-u\kappa  \right)u\kappa \tau +\dot{D}u\tau  \right) \\ 
				& \hspace{30pt} +C\left( 1-u\kappa  \right)\left( -C\left( 1-u\kappa  \right)u{{\tau }^{2}}+C\kappa {{\left( 1-u\kappa  \right)}^{2}}+\dot{D}\left( 1-u\kappa  \right) \right)=0 \\ 
			\end{split}
		\end{equation*}
		holds.
		\item $\beta (t)={{\varphi }_{(\alpha ,N)}}\left( s(t),u(t) \right)$ is a line of curvature if and only if
		\begin{equation*}
			\frac{{{C}^{2}}}{D}=\frac{Cu\left( u{\kappa }'\tau +\left( 1-u\kappa  \right){\tau }' \right)-D\tau }{\tau \sqrt{{{\left( 1-u\kappa  \right)}^{2}}+{{u}^{2}}{{\tau }^{2}}}}
		\end{equation*}
		holds.
	\end{enumerate}
	
	\textit{Case 3}: Let $s={{s}_{0}}$ be constant. Then, $C=\dot{s}=0$ and we get that ${{\vec{v}}_{p}}={{\vec{q}}_{o}}$, i.e., $\beta (t)$ is a ruling. Then, we have followings:
	\begin{equation*}
		{{k}_{n}}=0, \hspace{4pt} {{\kappa }_{g}}=0, \hspace{4pt} {{\tau }_{g}}=-\frac{{{D}^{2}}\mu \sin \theta }{{{f}^{2}}+{{g}^{2}}} 
	\end{equation*}
	which give us
	\begin{enumerate}[(i)]
		\item All rulings are asymptotic.
		\item All rulings are geodesic.
		\item The ruling $\beta (t)={{\varphi }_{(\alpha ,{{q}_{o}})}}\left( {{s}_{0}},u(t) \right)$is a line of curvature if and only if $\mu \sin \theta =0$ which suggests that either ${{\varphi }_{(\alpha ,{{q}_{o}})}}={{\varphi }_{(\alpha ,T)}}$ or $\alpha$ is a plane curve.
	\end{enumerate}
	
	\textit{Case 4}: Let $u={{u}_{0}}$ be constant. Then, $D=\dot{u}=0$ and we have followings:
	\begin{equation*}
		\begin{split}
			{{k}_{n}}&=\frac{{{C}^{2}}}{\sqrt{{{f}^{2}}+{{g}^{2}}}}\left[ \left( f\mu +g\eta  \right)\cos \theta -{{g}^{2}}{{\left( \frac{f}{g} \right)}_{\hspace{-2pt}s}}-\left( {{f}^{2}}+{{g}^{2}} \right)\xi  \right]\\
			{{\kappa }_{g}}&=\frac{1}{\sqrt{{{f}^{2}}+{{g}^{2}}}}\left[ C\cos \theta \left( -\left( {{f}^{2}}+{{g}^{2}} \right)\dot{C} \right. \right. \\ 
			& \hspace{30pt} \left. +C\cos \theta \left( \eta f-\mu g \right)-\frac{1}{2}C{{\left( {{f}^{2}}+{{g}^{2}} \right)}_{s}} \right) \\ 
			& \hspace{30pt} +Cg\left( \dot{C}g\cos \theta -Cg{\theta }'\sin \theta -\mu C{{g}^{2}}+fC\eta g \right) \\ 
			& \hspace{30pt} \left. \,+Cf\left( \dot{C}f\cos \theta -Cf{\theta }'\sin \theta -\mu fgC+C{{f}^{2}}\eta  \right) \right] \\
			{{\tau }_{g}}&=\frac{{{C}^{2}}}{{{f}^{2}}+{{g}^{2}}}\left[ \left( {{f}^{2}}+{{g}^{2}} \right)(f\mu +g\eta +\xi \cos \theta )+{{g}^{2}}\cos \theta {{\left( \frac{f}{g} \right)}_{\hspace{-2pt}s}} \right]
		\end{split}
	\end{equation*}
	
	\begin{enumerate} [(i)]
		\item The parameter curve $\beta (t)={{\varphi }_{(\alpha ,{{q}_{o}})}}\left( s(t),{{u}_{0}} \right)$ is an asymptotic curve if and only if
		\begin{equation*}
			\left( f\mu +g\eta  \right)\cos \theta -{{g}^{2}}{{\left( \frac{f}{g} \right)}_{\hspace{-2pt}s}}-\left( {{f}^{2}}+{{g}^{2}} \right)\xi =0
		\end{equation*}
		holds.
		\item The parameter curve $\beta (t)={{\varphi }_{(\alpha ,{{q}_{o}})}}\left( s(t),{{u}_{0}} \right)$is a geodesic if and only if
		\begin{equation*}
			\begin{split}
				& C\cos \theta \left( -\left( {{f}^{2}}+{{g}^{2}} \right)\dot{C} + C\cos \theta \left( \eta f-\mu g \right)-\frac{1}{2}C{{\left( {{f}^{2}}+{{g}^{2}} \right)}_{s}} \right) \\ 
				& \hspace{15pt} +Cg\left( \dot{C}g\cos \theta -Cg{\theta }'\sin \theta -\mu C{{g}^{2}}+fC\eta g \right) \\ 
				& \hspace{15pt} +Cf\left( \dot{C}f\cos \theta -Cf{\theta }'\sin \theta -\mu fgC+C{{f}^{2}}\eta  \right)=0 \\ 
			\end{split}
		\end{equation*}
		holds.
		\item The parameter curve $\beta (t)={{\varphi }_{(\alpha ,{{q}_{o}})}}\left( s(t),{{u}_{0}} \right)$is a line of curvature if and only if
		\begin{equation*}
			\left( {{f}^{2}}+{{g}^{2}} \right)(f\mu +g\eta +\xi \cos \theta )+{{g}^{2}}\cos \theta {{\left( \frac{f}{g} \right)}_{\hspace{-2pt}s}}=0
		\end{equation*}
		holds.
	\end{enumerate}
	
	\textit{Case 5}: Let $C=\dot{s}$, $D=\dot{u}$ be non-zero constants. Then, the curve has the parametric form $\beta (t)={{\varphi }_{(\alpha ,{{q}_{o}})}}\left( {{c}_{1}}t+{{c}_{2}},{{d}_{1}}t+{{d}_{2}} \right)$ where ${{c}_{i}}, {{d}_{i}}, (i=1,2)$ are constants and we have
	\begin{equation*}
		\begin{split}
			{{k}_{n}}&=C\left[ \left( C\cos \theta +D \right)\left( {{A}_{1}}\cos \theta +{{A}_{2}} \right)+\left( C{{A}_{1}}+D{{B}_{1}} \right)\left( {{f}^{2}}+{{g}^{2}} \right) \right]\\
			{{\kappa }_{g}}&=\frac{1}{\sqrt{{{f}^{2}}+{{g}^{2}}}}\left[ \left( C\cos \theta +D \right) \right.\left( C\left( \eta f-\mu g \right)\left( 2D+\cos \theta  \right)-\frac{1}{2}C{{\left( {{f}^{2}}+{{g}^{2}} \right)}_{s}} \right) \\ 
			& \hspace{15pt} +Cg\left( -Cg{\theta }'\sin \theta -\mu C{{g}^{2}}+fC\eta g \right)\left. \,+Cf\left( -Cf{\theta }'\sin \theta -\mu fgC+C{{f}^{2}}\eta  \right) \right],\\
			{{\tau }_{g}}&=\sqrt{{{f}^{2}}+{{g}^{2}}}\left[ C\left( C{{A}_{2}}+D{{B}_{2}} \right)-D\left( C{{A}_{1}}+D{{B}_{1}} \right) \right]		
		\end{split}
	\end{equation*}
	which give followings:
	\begin{enumerate} [(i)]
		\item $\beta (t)={{\varphi }_{(\alpha ,{{q}_{o}})}}\left( {{c}_{1}}t+{{c}_{2}},{{d}_{1}}t+{{d}_{2}} \right)$ is an asymptotic curve if and only if
		\begin{equation*}
			\left( C\cos \theta +D \right)\left( {{A}_{1}}\cos \theta +{{A}_{2}} \right)+\left( C{{A}_{1}}+D{{B}_{1}} \right)\left( {{f}^{2}}+{{g}^{2}} \right)=0
		\end{equation*}
		holds.
		\item $\beta (t)={{\varphi }_{(\alpha ,{{q}_{o}})}}\left( {{c}_{1}}t+{{c}_{2}},{{d}_{1}}t+{{d}_{2}} \right)$ is a geodesic if and only if
		\begin{equation*}
			\begin{split}
			&\left( C\cos \theta +D \right)\left( +C\left( \eta f-\mu g \right)\left( 2D+\cos \theta  \right)-\frac{1}{2}C{{\left( {{f}^{2}}+{{g}^{2}} \right)}_{s}} \right) \\ 
			& \hspace{15pt} +Cg\left( -Cg{\theta }'\sin \theta -\mu C{{g}^{2}}+fC\eta g \right)+Cf\left( -Cf{\theta }'\sin \theta -\mu fgC+C{{f}^{2}}\eta  \right)=0 \\ 
			\end{split}
		\end{equation*}
		holds.
		\item $\beta (t)={{\varphi }_{(\alpha ,{{q}_{o}})}}\left( {{c}_{1}}t+{{c}_{2}},{{d}_{1}}t+{{d}_{2}} \right)$ is a line of curvature if and only if
		\begin{equation*}
			\frac{C{{A}_{1}}+D{{B}_{1}}}{C{{A}_{2}}+D{{B}_{2}}}=\text{constant}
		\end{equation*}
		holds.
	\end{enumerate}
	
	Let now consider the Frenet frame of a non-cylindrical OT-ruled surface $\varphi_{(\alpha, q_o)}$. Differentiating the ruling ${{\vec{q}}_{o}}=\cos \theta \vec{T}+\sin \theta \vec{N}$,  it follows
	\begin{equation} \label{derivativeofruling} 
		\vec{q}\hspace{2pt}'_{\hspace{-2pt} o} = -\eta \sin \theta \,\vec{T}+\eta \cos \theta \,\vec{N}+\tau \sin \theta \,\vec{B}
	\end{equation}
	Then, the central normal and central tangent vectors of OT-ruled surface $\varphi_{(\alpha, q_o)}$ are computed as 
	\begin{equation} \label{handavectors}
		\begin{split}
		\vec{h}&=\frac{1}{\sqrt{{{\eta }^{2}}+{{\tau }^{2}}{{\sin }^{2}}\theta }}\left( -\eta \sin \theta \,\vec{T}+\eta \cos \theta \,\vec{N}+\tau \sin \theta \,\vec{B} \right) \\ 
		\vec{a}&=\frac{1}{\sqrt{{{\eta }^{2}}+{{\tau }^{2}}{{\sin }^{2}}\theta }}\left( \tau {{\sin }^{2}}\theta \,\vec{T}-\tau \cos \theta \sin \theta \,\vec{N}+\eta \vec{B} \right) \\ 
		\end{split}
	\end{equation}
	respectively. From the equations (\ref{derivativeofruling}) and (\ref{handavectors}), we have following theorem:
	
	\begin{theorem}
		For the OT-ruled surface $\varphi_{(\alpha, q_o)}$ the followings are equivalent:
		\begin{enumerate} [(i)]
			\item The angle between the vectors ${{\vec{q}}_{o}}$ and $\vec{T}$ is given by $\theta =-\int_{0}^{s}{\kappa ds}$.
			\item The central normal vector $\vec{h}$ coincides with the binormal vector $\vec{B}$ of $\alpha$.
			\item The central tangent vector $\vec{a}$ lies in the osculating plane of $\alpha$.
		\end{enumerate}
	\end{theorem}
	\begin{proof}
		Let the angle $\theta $ be given by $\theta =-\int_{0}^{s}{\kappa d}s$. Then, we get $\eta =0$. Thus, the proof is clear from (\ref{handavectors}).
	\end{proof}
	
	\begin{corollary}
		Let the angle between the vectors ${{\vec{q}}_{o}}$ and $\vec{T}$ is given by $\theta =-\int_{0}^{s}{\kappa d}s$. Then, $\alpha $ is a general helix if and only if the OT-ruled surface ${{\varphi }_{(\alpha ,{{q}_{o}})}}$ is an $h$-slant ruled surface.
	\end{corollary}
	
	\begin{theorem}
		The Frenet frame $\left\{ {{{\vec{q}}}_{o}},\vec{h},\vec{a} \right\}$ of OT-ruled surface ${{\varphi }_{(\alpha ,{{q}_{o}})}}$ coincides with the Frenet frame $\left\{ \vec{T},\vec{N},\vec{B} \right\}$ of base curve $\alpha $ if and only if ${{\varphi }_{(\alpha ,{{q}_{o}})}}$ is the tangent surface ${{\varphi }_{(\alpha ,T)}}$ of $\alpha$.
	\end{theorem}
	
	\section{Examples}
	
	\begin{example}
		Let consider the general helix curve ${{\alpha }_{1}}$ given by the parametrization
		\begin{equation*}
			{{\vec{\alpha }}_{1}}(s)=\left( \cos \left( \frac{s}{\sqrt{2}} \right),\sin \left( \frac{s}{\sqrt{2}} \right),\frac{s}{\sqrt{2}} \right)
		\end{equation*}
		For the required Frenet elements of ${{\alpha }_{1}}$, we obtain
		\begin{equation*}
			\begin{split}
				& \vec{T}(s)=\left( -\frac{1}{\sqrt{2}}\sin \left( \frac{s}{\sqrt{2}} \right),\frac{1}{\sqrt{2}}\cos \left( \frac{s}{\sqrt{2}} \right),\frac{1}{\sqrt{2}} \right), \hspace{4pt} \vec{N}(s)=\left( -\cos \left( \frac{s}{\sqrt{2}} \right),-\sin \left( \frac{s}{\sqrt{2}} \right),0 \right), \\ 
				& \kappa (s)=\frac{1}{2}, \hspace{4pt} \tau (s)=\frac{1}{2}.
			\end{split}			
		\end{equation*}
		By choosing $\theta (s)=s$, we get
		\begin{equation*}
			\begin{split}
				{{{\vec{q}}}_{o}}(s)&=\left( -\frac{1}{\sqrt{2}}\cos (s)\sin \left( \frac{s}{\sqrt{2}} \right)-\sin (s)\cos \left( \frac{s}{\sqrt{2}} \right) \right., \\ 
				& \hspace{20pt} \left. \frac{1}{\sqrt{2}}\cos (s)\cos \left( \frac{s}{\sqrt{2}} \right)-\sin (s)\sin \left( \frac{s}{\sqrt{2}} \right),\,\,\frac{1}{\sqrt{2}}\cos (s) \right).\\ 
			\end{split}
		\end{equation*}
		and the OT-ruled surface ${{\varphi }_{_{1}({{\alpha }_{1}},{{q}_{o}})}}$ has the parametrization
		\begin{equation*}
			\begin{split}
				{{{\vec{\varphi }}}_{_{1}({{\alpha }_{1}},{{q}_{o}})}}&=\left( \cos \left( \frac{s}{\sqrt{2}} \right)+u\left( -\frac{1}{\sqrt{2}}\cos (s)\sin \left( \frac{s}{\sqrt{2}} \right)-\sin (s)\cos \left( \frac{s}{\sqrt{2}} \right) \right) \right., \\ 
				& \hspace{20pt} \sin \left( \frac{s}{\sqrt{2}} \right)+u\left( \frac{1}{\sqrt{2}}\cos (s)\cos \left( \frac{s}{\sqrt{2}} \right)-\sin (s)\sin \left( \frac{s}{\sqrt{2}} \right) \right), \\ 
				& \hspace{20pt} \left. \frac{s}{\sqrt{2}}+\frac{1}{\sqrt{2}}u\cos (s) \right).\\ 
			\end{split}
		\end{equation*}
		From (\ref{strictionline}), the equation of the striction line of OT-ruled surface ${{\varphi }_{_{1}({{\alpha }_{1}},{{q}_{o}})}}$ is given by
		\begin{equation*}
			\begin{split}
				{{{\vec{c}}}_{1}}(s)&=\left( \frac{\frac{3\sqrt{2}}{2}\sin (2s)\sin \left( \frac{\sqrt{2}}{2}s \right)-\cos \left( \frac{\sqrt{2}}{2}s \right)\left( 5{{\cos }^{2}}(s)+4 \right)}{{{\cos }^{2}}(s)-10}, \right. \\ 
				& \hspace{20pt} -\frac{\frac{3\sqrt{2}}{2}\sin (2s)\cos \left( \frac{\sqrt{2}}{2}s \right)+\sin \left( \frac{\sqrt{2}}{2}s \right)\left( 5{{\cos }^{2}}(s)+4 \right)}{{{\cos }^{2}}(s)-10}, \\ 
				& \hspace{20pt} \left. \frac{\sqrt{2}\left( s{{\cos }^{2}}(s)-3\sin (2s)-10s \right)}{2\left( {{\cos }^{2}}(s)-10 \right)} \right). \\ 
			\end{split}
		\end{equation*}
		The curvatures of ${{\varphi }_{_{1}({{\alpha }_{1}},{{q}_{o}})}}$ are computed as $\eta (s)=\frac{3}{2}$, $\xi (s)=\frac{1}{2}\cos (s)$, $\mu (s)=\frac{1}{2}\sin (s)$ and the functions $f$ and $g$ are given by $f(s,u)=\sin (s)+\frac{3}{2}u$, $g(s,u)=\frac{1}{2}u\sin (s)$. The graph of ${{\varphi }_{_{1}({{\alpha }_{1}},{{q}_{o}})}}$ for the intervals $s\in \left[ 0,3\pi  \right]$, $u\in \left[ -1,1 \right]$ is given in Figure \ref{fig1}. From Proposition \ref{strictionlineprop}, the base curve ${{\alpha }_{1}}$ (red) and striction line ${{c}_{1}}$ (blue) intersect at the points ${{\varphi }_{_{1}({{\alpha }_{1}},{{q}_{o}})}} (0,0)$, ${{\varphi }_{_{1}({{\alpha }_{1}},{{q}_{o}})}} (\pi,0)$, ${{\varphi }_{_{1}({{\alpha }_{1}},{{q}_{o}})}} (2\pi,0)$, ${{\varphi }_{_{1}({{\alpha }_{1}},{{q}_{o}})}} (3\pi,0)$ which are also singular points of ${{\varphi }_{_{1}({{\alpha }_{1}},{{q}_{o}})}}$ and shown with black color in Figure \ref{fig1}.
		
		\begin{figure}[!h]
			\centering
			\begin{minipage}[b]{0.49\textwidth}
				\centering
				\includegraphics[width=\textwidth]{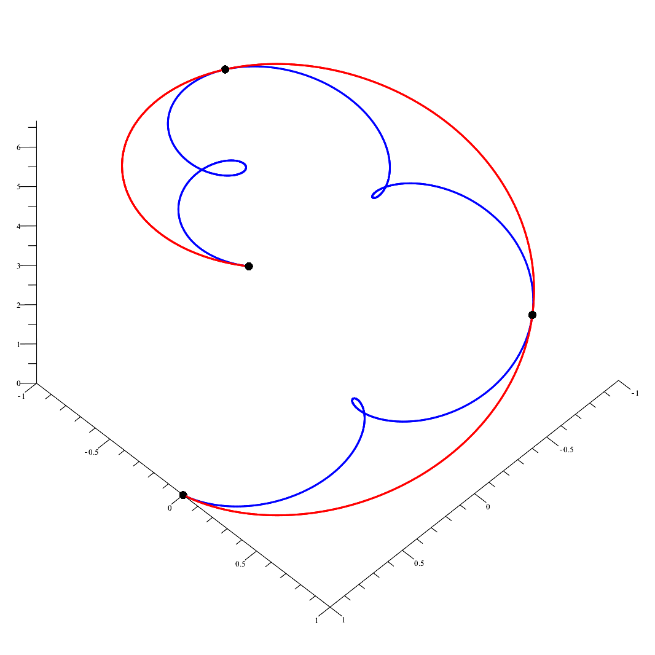}
			\end{minipage}
			\hfill
			\begin{minipage}[b]{0.49\textwidth}
				\centering
				\includegraphics[width=\textwidth]{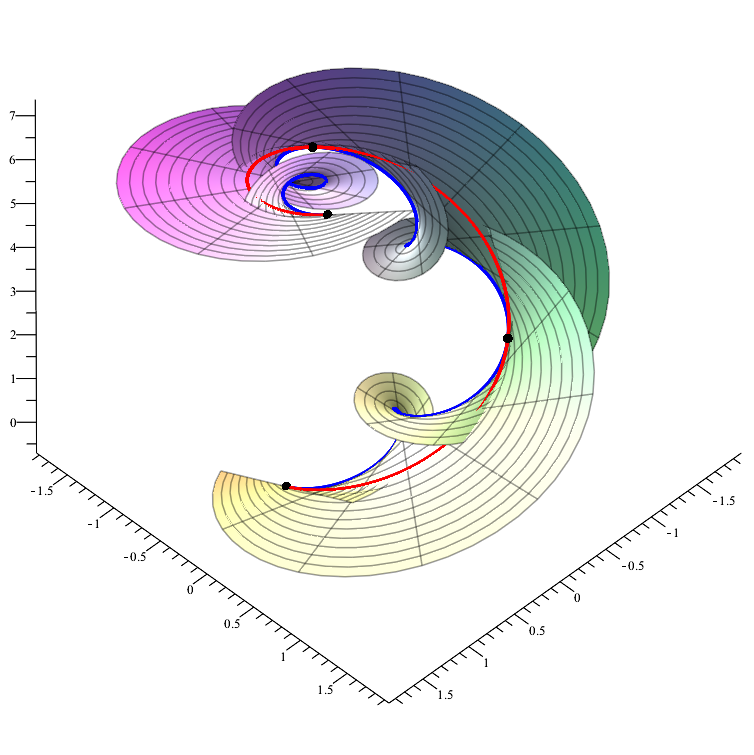}
			\end{minipage}
			\caption{The OT-ruled surface ${{\varphi }_{_{1}({{\alpha }_{1}},{{q}_{o}})}}$}
			\label{fig1}
		\end{figure}
	\end{example}
	
	\begin{example}
		Let the curve ${{\alpha }_{2}}$ be given by the parametrization
		\begin{equation*}
			{{\vec{\alpha }}_{2}}(s)=\left( \frac{3}{2}\cos \left( \frac{s}{2} \right)+\frac{1}{6}\cos \left( \frac{3s}{2} \right),\frac{3}{2}\sin \left( \frac{s}{2} \right)+\frac{1}{6}\sin \left( \frac{3s}{2} \right),\sqrt{3}\cos \left( \frac{s}{2} \right) \right)
		\end{equation*}
		whose required Frenet elements are
		\begin{equation*}
			\begin{split}
				\vec{T}(s)&=\left( -\frac{3}{4}\sin \left( \frac{s}{2} \right)-\frac{1}{4}\sin \left( \frac{3s}{2} \right),\frac{3}{4}\cos \left( \frac{s}{2} \right)+\frac{1}{4}\cos \left( \frac{3s}{2} \right),-\frac{\sqrt{3}}{2}\sin \left( \frac{s}{2} \right) \right),\\
				\vec{N}(s)&=\left( -\frac{\sqrt{3}}{2}\cos (s),-\frac{\sqrt{3}}{2}\sin (s),-\frac{1}{2} \right),\\
				\kappa (s)&=\frac{\sqrt{3}}{2}\cos \left( \frac{s}{2} \right), \hspace{4pt} \tau (s)=-\frac{\sqrt{3}}{2}\sin \left( \frac{s}{2} \right),
			\end{split}
		\end{equation*}
		where we calculate
		\begin{equation*}
			\frac{{{\kappa }^{2}}}{{{\left( {{\kappa }^{2}}+{{\tau }^{2}} \right)}^{3/2}}}{{\left( \frac{\tau }{\kappa } \right)}^{\prime }}=-\frac{\sqrt{3}}{3}=\textnormal{constant}
		\end{equation*}
		Therefore, we obtain that ${{\alpha}_{2}}$ is a slant helix. By choosing $\theta (s)=\frac{s}{2}$, we get
		\begin{equation*}
			\begin{split}
			{{{\vec{q}}}_{o}}(s)&=\left( -\frac{1}{2}\sin \left( \frac{s}{2} \right)\left( 2{{\cos }^{2}}\left( \frac{s}{2} \right)\left( \cos \left( \frac{s}{2} \right)+\sqrt{3} \right)+\cos \left( \frac{s}{2} \right)-\sqrt{3} \right), \right. \\ 
			& \hspace{20pt} \left. \cos \left( \frac{s}{2} \right)\left( {{\cos }^{2}}\left( \frac{s}{2} \right)\left( \sqrt{3}+\cos \left( \frac{s}{2} \right) \right)-\sqrt{3} \right),-\frac{1}{2}\sin \left( \frac{s}{2} \right)\left( \sqrt{3}\cos \left( \frac{s}{2} \right)+1 \right) \right). \\ 
			\end{split}			
		\end{equation*}
		Then, the parametrization of the OT-ruled surface ${{\varphi }_{_{2}({{\alpha }_{2}},{{q}_{o}})}}$ and its striction line ${{c}_{2}}$ can be written easily by using the equalities (\ref{OT-ruledSurfaceEquation}) and (\ref{strictionline}), respectively. The curvatures of that surface are
		\begin{equation*}
			\eta (s)=\frac{1}{2}+\frac{\sqrt{3}}{2}\cos \left( \frac{s}{2} \right), \hspace{4pt} \xi (s)=-\frac{\sqrt{3}}{4}\sin (s), \hspace{4pt} \mu (s)=-\frac{\sqrt{3}}{2}{{\sin }^{2}}\left( \frac{s}{2} \right). 
		\end{equation*}
		Furthermore, the functions $f$ and $g$ are calculated as
		\begin{equation*}
			f(s,u)=\sin \left( \frac{s}{2} \right)+u\left( \frac{1}{2}+\frac{\sqrt{3}}{2}\cos \left( \frac{s}{2} \right) \right), \hspace{4pt} g(s,u)=-\frac{\sqrt{3}}{2}u{{\sin }^{2}}\left( \frac{s}{2} \right).
		\end{equation*}
		The graph of ${{\varphi }_{_{2}({{\alpha }_{2}},{{q}_{o}})}}$ for intervals $s\in \left[ -2\pi ,2\pi  \right]$ and $u\in \left[ -1,1 \right]$ is given in Figure \ref{fig2}. From Proposition \ref{strictionlineprop}, the base curve ${{\alpha }_{2}}$ (red) and striction line ${{c}_{2}}$ (blue) intersect at the points
		\begin{equation*}
			\begin{split}
				p_1 &= {{\varphi }_{_{2}({{\alpha }_{2}},{{q}_{o}})}} (-2\pi,0) = {{\varphi }_{_{2}({{\alpha }_{2}},{{q}_{o}})}} (\pi,0), \hspace{4pt} p_2 = {{\varphi }_{_{2}({{\alpha }_{2}},{{q}_{o}})}} (0,0)\\
				p_3 &= {{\varphi }_{_{2}({{\alpha }_{2}},{{q}_{o}})}} \left( 2\left( \pi - \arccos \left( \frac{\sqrt{3}}{3} \right)  \right), 0  \right) , \hspace{4pt} p_4 = {{\varphi }_{_{2}({{\alpha }_{2}},{{q}_{o}})}} \left( 2\left( \pi + \arccos \left( \frac{\sqrt{3}}{3} \right)  \right), 0  \right).
			\end{split}
		\end{equation*}
		Here, ${{p}_{1}},{{p}_{2}}\in S$ are singular points of ${{\varphi }_{_{2}({{\alpha }_{2}},{{q}_{o}})}}$ ${{p}_{3}},{{p}_{4}}\in Y$ are non-singular points which are given black and green in Figure \ref{fig2}, respectively.
		
		\begin{figure}[!h]
			\centering
			\begin{minipage}[b]{0.49\textwidth}
				\centering
				\includegraphics[width=\textwidth]{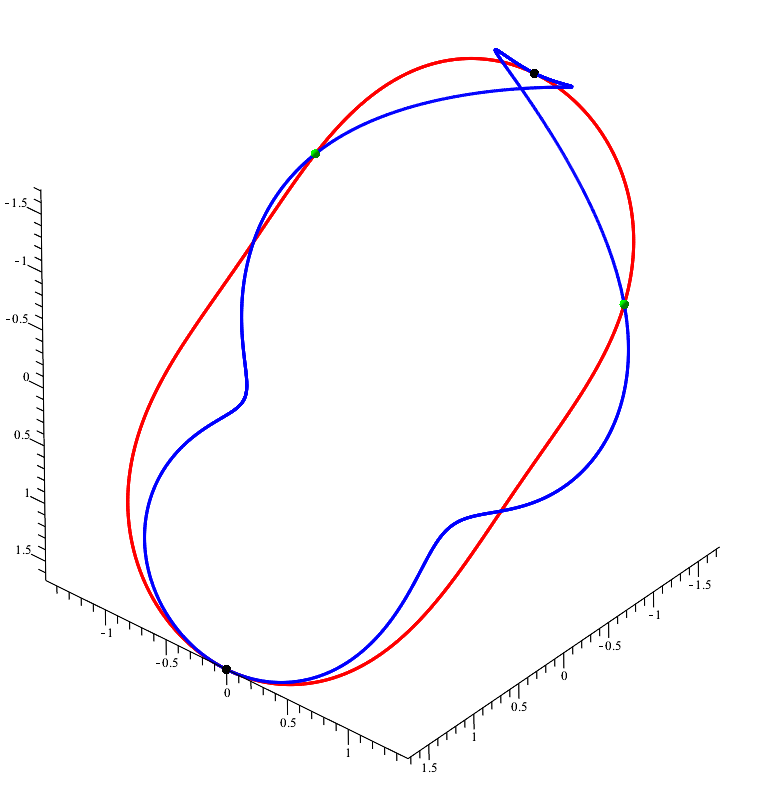}
			\end{minipage}
			\hfill
			\begin{minipage}[b]{0.49\textwidth}
				\centering
				\includegraphics[width=\textwidth]{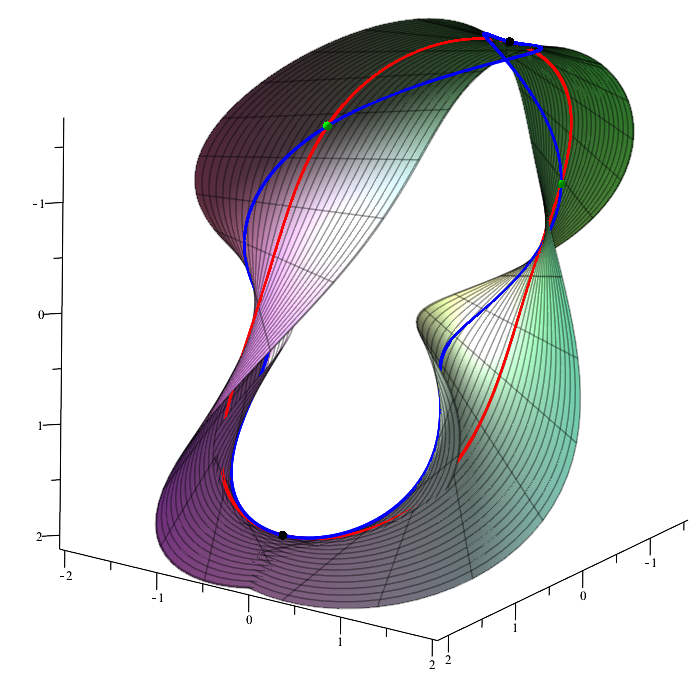}
			\end{minipage}
			\caption{The OT-ruled surface ${{\varphi }_{_{2}({{\alpha }_{2}},{{q}_{o}})}}$}
			\label{fig2}
		\end{figure}
	\end{example}
	
	\begin{example}
		Let ${{\alpha }_{3}}$ be given by the parametrization
		\begin{equation*}
			\begin{split}
				{{{\vec{\alpha }}}_{3}}(s)&=\frac{5\sqrt{26}}{26}\left( \frac{\left( \sqrt{26}-26 \right)\sin \left( \left( 1+\frac{\sqrt{26}}{13} \right)s \right)}{104+8\sqrt{26}}+\frac{\left( \sqrt{26}+26 \right)\sin \left( \left( 1-\frac{\sqrt{26}}{13} \right)s \right)}{-104+8\sqrt{26}}-\frac{1}{2}\sin (s) \right., \\ 
				& \hspace{50pt} \frac{\left( 26-\sqrt{26} \right)\cos \left( \left( 1+\frac{\sqrt{26}}{13} \right)s \right)}{104+8\sqrt{26}}-\frac{\left( \sqrt{26}+26 \right)\cos \left( \left( 1-\frac{\sqrt{26}}{13} \right)s \right)}{-104+8\sqrt{26}}+\frac{1}{2}\cos (s),\, \\ 
				& \hspace{50pt} \left. \frac{5}{4}\cos \left( \frac{\sqrt{26}}{13}s \right) \vphantom{\frac{\left( \sqrt{26}-26 \right)\sin \left( \left( 1+\frac{\sqrt{26}}{13} \right)s \right)}{104+8\sqrt{26}}} \right) \\ 
			\end{split}
		\end{equation*}
		which is a special chosen of general Salkowski curve defined in \cite{Monterde}. The required Frenet elements are
		\begin{equation*}
			\begin{split}
				\vec{T}(s)&=\left( -\cos (s)\cos \left( \frac{\sqrt{26}}{26}s \right)-\frac{\sqrt{26}}{26}\sin (s)\sin \left( \frac{\sqrt{26}}{26}s \right) \right., \\ 
				& \hspace{20pt} \left. -\sin (s)\cos \left( \frac{\sqrt{26}}{26}s \right)+\frac{\sqrt{26}}{26}\cos (s)\sin \left( \frac{\sqrt{26}}{26}s \right),-\frac{5\sqrt{26}}{26}\sin \left( \frac{\sqrt{26}}{26}s \right) \right) \\ 
				\vec{N}(s)&=\left( \frac{5\sqrt{26}}{26}\sin (s),-\frac{5\sqrt{26}}{26}\cos (s),-\frac{\sqrt{26}}{26} \right),\\
				\kappa (s)&=1, \hspace{4pt} \tau (s)=\tan \left( \frac{\sqrt{26}}{26}s \right).
			\end{split}
		\end{equation*}
		By choosing $\theta (s)=\frac{s}{\sqrt{26}}$, we get
		\begin{equation*}
			\begin{split}
				{{{\vec{q}}}_{o}}(s)&=\left( -\frac{\sqrt{26}}{26}\cos \left( \frac{\sqrt{26}}{26}s \right)\sin (s)\sin \left( \frac{\sqrt{26}}{26}s \right)-\cos (s){{\cos }^{2}}\left( \frac{\sqrt{26}}{26}s \right)+\frac{5\sqrt{26}}{26} \right.\sin (s)\sin \left( \frac{\sqrt{26}}{26}s \right), \\ 
				& \hspace{22pt} \frac{\sqrt{26}}{26}\cos \left( \frac{\sqrt{26}}{26}s \right)\cos (s)\sin \left( \frac{\sqrt{26}}{26}s \right)-\frac{5\sqrt{26}}{26}\cos (s)\sin \left( \frac{\sqrt{26}}{26}s \right)-\sin (s){{\cos }^{2}}\left( \frac{\sqrt{26}}{26}s \right), \\ 
				& \hspace{20pt} \left. -\frac{\sqrt{26}}{26}\sin \left( \frac{\sqrt{26}}{26}s \right)\left( 5\cos \left( \frac{\sqrt{26}}{26}s \right)+1 \right) \right) \\
			\end{split}
		\end{equation*}
		Then the parametrization of the OT-ruled surface ${{\varphi }_{_{3}({{\alpha }_{3}},{{q}_{o}})}}$ and the equation of striction line ${{c}_{3}}$ can be written easily from the equalities (\ref{OT-ruledSurfaceEquation}) and (\ref{strictionline}), respectively. This surface has the curvatures
		\begin{equation*}
			\eta (s)=1+\frac{\sqrt{26}}{26}, \hspace{4pt} \xi (s)=\sin \left( \frac{\sqrt{26}}{26}s \right), \hspace{4pt} \mu (s)=\tan \left( \frac{\sqrt{26}}{26}s \right)\sin \left( \frac{\sqrt{26}}{26}s \right),
		\end{equation*}
		and the functions $f$ and $g$ are calculated as
		\begin{equation*}
			f(s,u)=\sin \left( \frac{\sqrt{26}}{26}s \right)+u\left( 1+\frac{\sqrt{26}}{26} \right), \hspace{4pt} g(s,u)=u\tan \left( \frac{\sqrt{26}}{26}s \right)\sin \left( \frac{\sqrt{26}}{26}s \right).
		\end{equation*}
		The graph of ${{\varphi }_{_{3}({{\alpha }_{3}},{{q}_{o}})}}$ for intervals $s\in \left[ -\frac{\sqrt{26}}{2}\pi ,\frac{\sqrt{26}}{2}\pi  \right]$ and $u \in \left[ -0.5,0.5 \right]$ is given in Figure \ref{fig3}. Proposition \ref{strictionlineprop}, the base curve ${{\alpha }_{3}}$ (red) and striction line ${{c}_{3}}$ (blue) intersect at the points ${{\varphi }_{_{3}({{\alpha }_{3}},{{q}_{o}})}} \left( -\frac{\sqrt{26}}{2}\pi, 0 \right) $, ${{\varphi }_{_{3}({{\alpha }_{3}},{{q}_{o}})}} \left( 0, 0 \right) $ and ${{\varphi }_{_{3}({{\alpha }_{3}},{{q}_{o}})}} \left( \frac{\sqrt{26}}{2}\pi, 0 \right) $. All these points are singular points of ${{\varphi }_{_{3}({{\alpha }_{3}},{{q}_{o}})}}$ and given by black in Figure \ref{fig3}.
		
		\begin{figure}[!h]
			\centering
			\begin{minipage}[b]{0.49\textwidth}
				\centering
				\includegraphics[width=\textwidth]{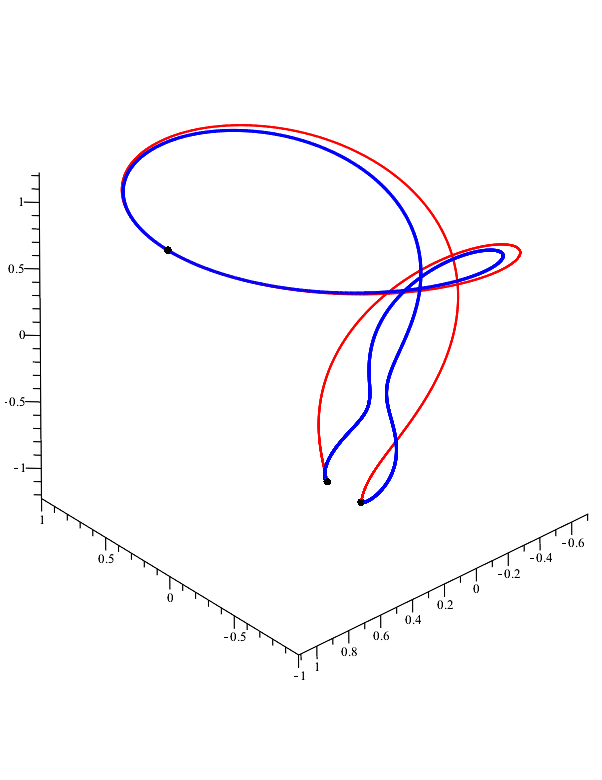}
			\end{minipage}
			\hfill
			\begin{minipage}[b]{0.49\textwidth}
				\centering
				\includegraphics[width=\textwidth]{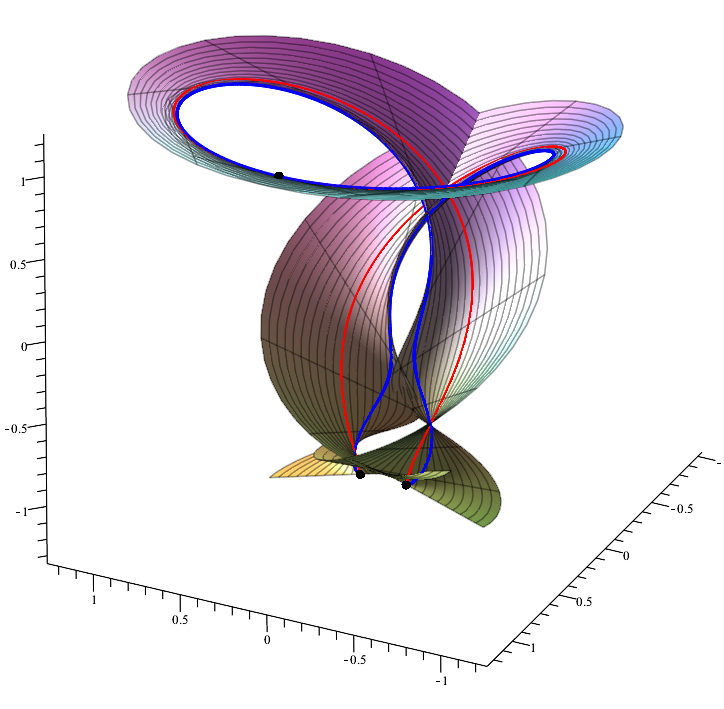}
			\end{minipage}
			\caption{The OT-ruled surface ${{\varphi }_{_{3}({{\alpha }_{3}},{{q}_{o}})}}$}
			\label{fig3}
		\end{figure}
	\end{example}
	
	\section{Conclusions}
	A new type of ruled surfaces has been defined according to the position of the ruling. Taking the ruling on the osculating plane of a curve, these surfaces is defined as osculating type ruled surface or OT-ruled surface. Many properties of such surfaces have been obtained. Of course, this subject can be considered in some other spaces such as Lorentzian space and Galilean space, and properties of OT-ruled surfaces can be given in these spaces according to the characters of base curve and ruling.

\end{document}